\documentclass[12pt,leqno]{amsart}
\usepackage{amssymb,amsthm,amsmath,latexsym}
\usepackage{xcolor}
\usepackage{setspace}
\usepackage{hyperref}
\usepackage{dcpic}
\usepackage{multicol}
\usepackage{pictexwd, dcpic}
\usepackage{enumerate}
\usepackage{graphicx}
\usepackage{tikz-cd}
\usepackage[all]{xy}
\usepackage[document]{ragged2e}
\usepackage{verbatim}
\usepackage{tikz}
\newtheorem{theorem}{\sc Theorem}[section]
\newtheorem{thm}[theorem]{\sc Theorem}
\newtheorem{lem}[theorem]{\sc Lemma}
\newtheorem{prop}[theorem]{\sc Proposition}
\newtheorem{cor}[theorem]{\sc Corollary}

\newtheorem{remar}[theorem]{\sc Remark}

\newtheorem*{thmA}{Theorem A}
\newtheorem*{corB}{Corollary B}
\newtheorem*{thmC}{Theorem C}
\newtheorem*{thmD}{Theorem D}

\begin{document}
\title[weak commutativity]{Structural results on the Weak Commutativity construction}
\author[Bastos]{R. Bastos}
\address{Departamento de Matem\'atica, Universidade de  Bras\'ilia,
Bras\'ilia-DF, 70910-900 Brazil \newline ORCID: https://orcid.org/0000-0002-5733-519X}
\email{(Bastos) bastos@mat.unb.br}
\author[de Oliveira]{R. de Oliveira}
\address{ Instituto de Matem\'atica e Estat\'istica, Universidade Federal de Goi\'as,
Goi\^ania-GO, 74690-900 Brazil \newline ORCID:https://orcid.org/0000-0003-3151-7304}
\email{(de Oliveira) ricardo@ufg.br}
\author[Ortega]{G. Ortega}
\address{Departamento de Matem\'atica, Universidade de Bras\'ilia,
Bras\'ilia-DF, 70910-900 Brazil \newline ORCID:https://orcid.org/0009-0004-4895-0555}
\email{(Ortega) ortega@mat.unb.br }
\keywords{Finiteness conditions; double Sidki construction; lower central series; derived series; Periodic groups}

\justify
\begin{abstract}
The weak commutativity group $\chi(G)$ is generated by two isomorphic groups $G$ and $G^{\varphi }$ subject to the relations $[g,g^{\varphi}]=1$ for all $g \in G$. We obtain new expressions for the terms of the derived series and the lower central series of $\chi(G)$. We also present new bounds for the exponent of some sections of $\chi(G)$. \vspace{0,3cm} 

\noindent {\bf Mathematics Subject Classification.} 20E06, 20F05, 20F14, 20F50
\end{abstract}

\maketitle

\section{Introduction}

Let  $G^{\varphi}$ be a isomorphic copy of the group $G$, via $\varphi : G \rightarrow G^{\varphi}$, given by $g \mapsto g^{\varphi}$. In \cite{Sidki}, Sidki introduces and analyzes the {\it weak commutativity group}, $\chi(G)$, which is defined as follows \[ \chi(G) = \langle G \cup G^{\varphi} \mid [g,g^{\varphi}]=1, \forall g \in G \rangle.\]

The group $\chi(G)$ is also known as the double Sidki of $G$. Now, we recall some homomorphisms associated with the group $\chi (G)$ to describe some of its normal subgroups and sections. First of all $\chi(G)$ maps onto $G$ by $g\mapsto g$, $g^{\varphi }\mapsto g$
with kernel $L(G)=\left\langle g^{-1}g^{\varphi } \mid g\;\in G\right\rangle $ and it
maps onto $G\times G$ by $g\mapsto \left( g,1\right) ,g^{\varphi
}\mapsto \left( 1,g\right) $ with kernel $D(G)= [G,G^{\varphi}] = \langle [g,h^{\varphi}] \mid g,h \in G \rangle$. As a consequence of the defining relation of the group $\chi(G)$, the subgroups $L(G)$ and $D(G)$ commute. Furthermore, if we denote by $T(G)$ the
subgroup of $G\times G\times G$ generated by $\{(g,g,1),(1,g,g)\mid g\in G\}$, then $\chi (G)$ maps onto $T(G)$ by $g\mapsto \left( g,g,1\right)$, $g^{\varphi }\mapsto \left( 1,g,g\right) $, with kernel $W(G)=L(G)\cap D(G)$. The subgroup $W(G)$ is central in $ L(G)D(G)$ and $\chi(G)/W(G)$ is isomorphic to a subgroup of $G \times G \times G$. Another normal subgroup of $\chi (G)$ is $R(G)=
[G,L(G),G^{\varphi}] \leq W(G)$. The weak commutativity group has some links with homology. It is worth pointing out that the section $W(G)/R(G)$ is isomorphic to the Schur multiplier $M(G)$ (cf. \cite[Lemma 4.1.11]{Sidki} and \cite[Lemma 2.2]{Roc82}). See also \cite{Miller}. Moreover, the section $D(G)/R(G)$ is isomorphic to the non-abelian exterior square $G \wedge G$. See \cite{BdMdO}, \cite{BdMGN}, \cite{BdOLN}, \cite{BK}, \cite{BK1}, \cite{DK}, \cite{GRS}, \cite{KM}, \cite{KS}, \cite{LO}, \cite{Mend}, \cite{Roc82} and \cite{Sidki} for a  rich source  of results  regarding this construction. 

In the present article we give more information on the structure of the group $\chi(G)$. The key approach was to describe a certain subgroup of $R(G)$.

The structure of the subgroup $R(G)$ is not yet completely understood, making it an obstacle to understanding the group $\chi(G)$. In \cite{Sidki}, Sidki described $R(G)$ for cases where $G$ is either an abelian or a perfect group. Specifically, he showed that if $G$ is an abelian group, then $R(G)^2=1$. Later, in \cite{KM}, Kochloukova and Mendon\c ca proved that if $G$ is a $2$-generated group, then $R(G)=1$ (see also \cite[Remark 3.2~(ii)]{BdOLN}). We obtain the following related results. 

\begin{thmA} \label{thm:R}
Let $G$ be a group. 
\begin{enumerate}[(a)]
    \item If $G$ is periodic, then $R(G)$ is periodic.
    \item If $G$ has finite exponent, then $\exp(R(G))$ divides $\exp(G) \cdot \exp(G')$.
\end{enumerate}
\end{thmA}

We do not know whether the bound obtained in Theorem A~(b) is sharp. At this moment, we have no examples that attain the bound in Theorem A~(b). See Remark \ref{rem:example} for more details. It is worth to mentioning that in the proof of Theorem A was crucial to understand the subgroup $L_{1,2}(G) := [L_1(G), L_2(G)] \leq R(G)$ (see Theorem \ref{TL12FG}, below). In the next result, we present more information on the subgroup $L_{1,2}(G)$. 

\begin{corB}
Let $G$ be a finitely generated periodic group. Then $L_{1,2}(G)$ is finite.     
\end{corB}

It is well-known that there are finitely generated periodic  groups which are infinite. Some examples with this property it has been constructed in \cite{Aleshin,Gri,G,GS,S}. Moreover, we present explicit connections between the order of the commutators and the periodicity of the subgroup $L_{1,2}(G)$ (see Theorem \ref{thm:expL12} and Remark \ref{rem:exp}, below).

In this direction, we also deduce new descriptions of the lower central and derived series of the group $\chi(G)$.  

\begin{thmC}\label{thm:derived}
Let $D = D(G)$, $L=L(G)$, $L_1 = [L,G]$, $L_2 = [L,G^{\varphi}]$ and $L_{1,2} = [L_1,L_2]$. Then   
\begin{enumerate}[(a)]
    \item The derived subgroup $\chi(G)'$ is the central product of  $D$ and  $L_1L_2$.
    \item The second derived subgroup $\chi(G)^{(2)}$ is given by the product of subgroups $D'L_1'L_2'L_{1,2}$.
    \item If $k \geq 2$, then $\chi(G)^{(k+1)}$ is the central product of $D^{(k)}$, $ L_1^{(k)}$ and $L_2^{(k)}$.
\end{enumerate}
\end{thmC}

Recall that a group $H$ is a central product of its normal subgroups $K_1, K_2, \ldots, K_l$ if $H=K_1 K_2 \cdots K_l$, $[K_i,K_j]=1$ for $i\neq j$ and $K_i \cap \left(\prod_{j\neq i}K_j\right) \leq Z(H)$ for all $i$.

\begin{thmD}
Let $n \geq 3$. Let $D = D(G)$, $L=L(G)$, $L_1 = [L,G]$ and $L_2 = [L,G^{\varphi}]$. Then 
\begin{enumerate}[(a)]
    \item The subgroups $[D,{}_{n-2}G]$, $[L_1,{}_{n-2} G]$, $[L_2,{}_{n-2} G^{\varphi}]$ are normal in $\chi(G)$. 
    \item The lower central series $\gamma_n\left(\chi(G)\right)=[D,_{n-2}G][L_1,_{n-2}G][L_2,_{n-2}G^{\varphi}]$.     
\end{enumerate}
\end{thmD}

We briefly describe the organization of the paper. In the next section we fix the notation and present certain commutator identities. In the third section we present a description to the derived subgroup of the group $\nu(G)$. This description will be useful in the proofs of the main results. Section 4 is devoted to the proof of Theorem A and Corollary B. The proofs of Theorems C and D are given in Section 5.  

\section{Preliminaries}

We follow the notations from \cite{Sidki} and review some basic properties of
the group $\chi(G)$. Let $G^{\varphi}$ be an isomorphic copy of $G$. By definition, 
\[
\chi(G) = \langle G \cup G^\varphi \mid  [g,g^\varphi]=1,\forall g \in G \rangle.
\]
Additionally set 
$L=L(G)$, $W=W(G)$, $R=R(G)$, $L_1=L_1(G)=[L,G]$, $L_2=L_2(G)=[L,G^\varphi]$, and $L_{1,2} = [L_1,L_2]$, we will omit the $G$ from these symbols whenever the
context allows it. 

As usual, for arbitrary elements $x,y \in G$ we write $x^y=y^{-1}xy$ for the conjugate of $x$ by $y$; the commutator of $x$ and $y$ is then $[x,y] = x^{-1}y^{-1}xy$ and our commutators are left normed: $[x,y,z] = [[x,y],z]$. To simplify notation, we write $[g,\varphi]$ instead of $g^{-1}g^{\varphi}$. 

The following basic properties are consequences of 
the defining relations of $\chi(G)$ and the commutator rules (see \cite[Proposition 4.1.13]{Sidki} and \cite[Lemma 2.1]{Roc82} for more details). 

\begin{lem} 
\label{basic.chi}
The following relations hold in $\chi(G)$, for all 
$x, y,y_i,z,z_i \in G$.
\begin{itemize}
\item[(i)] $[x,y^{\varphi}] = [x^{\varphi},y]$;
\item[(ii)] $[x,y^{\varphi}]^{z^{\varphi}} = [x,y^{\varphi}]^z$;
\item[(iii)] $[x,y^{\varphi}]^{\omega(z_1^{\varepsilon_1},\ldots,z_n^{\varepsilon_n})} = [x,y^{\varphi}]^{\omega(z_1,\ldots,z_n)}$ for $\varepsilon_i \in \{1,\varphi\}$ and any word $\omega(z_1,\ldots,z_n) \in G$; 
\item[(iv)] $[x^{\varphi},y,x]=[x,y,x^{\varphi}]$;
\item[(v)] $[x^{\varphi},y_1,\ldots,y_n,x] = [x,y_1,\ldots,y_n,x^{\varphi}]$.
\end{itemize}
\end{lem}

\begin{lem}\cite[Lemma 4.1.6]{Sidki}\label{l1} 
    Let $a,b,g,h,m\in G$, then 
    \begin{itemize}
        \item[(i)] $[h,g^{-\varphi}g]=[g^{\varphi},h][h,g]=[h,g][g^{\varphi},h] $;
        \item[(ii)] $ [h^{-1}h^{\varphi},g^{\varphi}]=[g^{\varphi},h][h^{\varphi},g^{\varphi}]=[h^{\varphi},g^{\varphi}][g^{\varphi},h]$.
    \end{itemize}
\end{lem}

The subgroup $L_{1,2}$ plays a central role in this work. To investigate some of its properties, we use \cite[Lemma 4.1.10]{Sidki},  which states that $[L_1,L_2]\leq L_1\cap L_2=W$. Therefore, $[L_1,L_2]$ centralizes both $L_1$ and $L_2$.

The next lemma shows that commutators are bilinear in their inputs and will be essential for the proof of Theorem C.

\begin{lem}\label{lem:L12prod}
    Let $\alpha_1,\alpha_2\in L_1$ and $\beta_1,\beta_2\in L_2$. Then we have 
    \[
    [\alpha_1\alpha_2,\beta_1\beta_2]= [\alpha_1,\beta_1][\alpha_1,\beta_2][\alpha_2,\beta_1][\alpha_2,\beta_2].
    \]
\end{lem}
\begin{proof}
    Once $L_{1,2}$ centralizes $L_1$ and $L_2$, we have that,
    \begin{align*}
        [\alpha_1\alpha_2,\beta_1\beta_2]=&[\alpha_1,\beta_1\beta_2]^{\alpha_2}[\alpha_2,\beta_1\beta_2]\\
        =&[\alpha_1,\beta_1\beta_2][\alpha_2,\beta_1\beta_2]\\
        =&[\alpha_1,\beta_2][\alpha_1,\beta_1]^{\beta_2}[\alpha_2,\beta_2][\alpha_2,\beta_1]^{\beta_2}\\
        =&[\alpha_1,\beta_2][\alpha_1,\beta_1][\alpha_2,\beta_2][\alpha_2,\beta_1]\\
        =&[\alpha_1,\beta_1][\alpha_1,\beta_2][\alpha_2,\beta_1][\alpha_2,\beta_2]. \qedhere
    \end{align*}
\end{proof}

The following results will be used on Theorem \ref{thm:expL12}. 

\begin{lem}\label{Lm:R} Let $g,h,a,b,c\in G$. Then, the following properties hold.
\begin{itemize}
    \item[(i)]$
 [\, [g,\varphi] [h,\varphi],a,b^{\varphi}]=[[g,\varphi] ,a,b^{\varphi}][ [h,\varphi],a,b^{\varphi}] \ \mod(L_{1,2})
$;
\item[(ii)] $ [\,[g,\varphi,a] [ h,\varphi,b] ,c^{\varphi}]=[g,\varphi,a ,c^{\varphi}][ h,\varphi,b ,c^{\varphi}]$
\end{itemize}
    
\end{lem}
\begin{proof}
{Let $g,h,a,b,c\in G$. For $(i)$ consider $x=[[g,\varphi] [h,\varphi],a,b^{\varphi}]$. Then, module $L_{1,2}$, we have that
\begin{align*}
    x&=[[[g,\varphi],a]^{[h,\varphi]}[[h,\varphi],a],b^{\varphi}]\\
    &=[[[g,\varphi],a]^{[h,\varphi]},b^{\varphi}]^{[[h,\varphi],a]}[[h,\varphi],a,b^{\varphi}]\\
    &=[[[g,\varphi],a]^{[h,\varphi]},b^{\varphi}][[h,\varphi],a,b^{\varphi}] \ \ \ \ \big( [[[g,\varphi],a]^{[h,\varphi]},b^{\varphi}]\in R \big)\\
    &= [[[g,\varphi],a],(b^{\varphi})^{[h,\varphi]^{-1}}]^{[h,\varphi]}[[h,\varphi],a,b^{\varphi}]\\
    &=[[[g,\varphi],a],b^{\varphi}[b^{\varphi},{[h,\varphi]^{-1}}]]^{[h,\varphi]}[[h,\varphi],a,b^{\varphi}]\\
    &=\left([[[g,\varphi],a],[b^{\varphi},{[h,\varphi]^{-1}}]][[g,\varphi],a,b^{\varphi}]^{[b^{\varphi},{[h,\varphi]^{-1}}]}\right)^{[h,\varphi]}[[h,\varphi],a,b^{\varphi}]\\
    &=[[g,\varphi],a,b^{\varphi}]^{[h,\varphi]}[[h,\varphi],a,b^{\varphi}] \ \ \ \ \big([[[g,\varphi],a],[b^{\varphi},{[h,\varphi]^{-1}}]]\in L_{1,2}\big)\\
    &=[[g,\varphi],a,b^{\varphi}][[h,\varphi],a,b^{\varphi}] \ \ \ \ \ \ \ \ \ 
    \ \ \big([[g,\varphi],a,b^{\varphi}]\in R\big)
\end{align*}
Now, to prove $(ii)$, observe that 
\begin{align*}
    [[[g,\varphi],a] [ [h,\varphi],b] ,c^{\varphi}]&=[[g,\varphi],a ,c^{\varphi}]^{[ [h,\varphi],b]}[ [h,\varphi],b, c^{\varphi}]\\
    &=[[g,\varphi],a ,c^{\varphi}][ [h,\varphi],b ,c^{\varphi}],
\end{align*}
since $ [[g,\varphi],a ,c^{\varphi}]\in R$.}
\end{proof}
The next is an immediate consequence of Lemma \ref{Lm:R}.  

\begin{lem}\label{Lm:R/L12} Let $G$ be a group. Then,
    \[
    \frac{R}{L_{1,2}}= \big\langle { [g,\varphi, a,b^{\varphi}]} L_{1,2} \mid g,a,b\in G\big\rangle.
    \]
\end{lem}

\begin{proof}
Let $g,a,b,g_1, \ldots ,g_n\in G$ and $\alpha_1, \ldots ,\alpha_r \in L_1$. Then, by Lemma \ref{Lm:R} item (ii)
\[
\left[ \prod_{i=1}^{r}\alpha_i,b^{\varphi}\right]= \prod_{i=1}^r[\alpha_i,b^{\varphi}].
\]
And, module $L_{1,2}$ we have that by Lemma \ref{Lm:R} item (i)
\[
\left[\prod_{i=1}^n[g_i,\varphi],a,b^{\varphi}\right]=\prod_{i=1}^n\left[g_i,\varphi,a,b^{\varphi}\right].
\]
Therefore, every element of $R/L_{1,2}$ splits in a product of elements of the form $[g,\varphi, a,b^{\varphi}] L_{1,2}$.
\end{proof}

\section{The group \texorpdfstring{$\nu(G)$}{ν(G)}}

The non-abelian tensor square $G \otimes G$ of a group $G$, as introduced by Brown and Loday \cite{BL}, is defined to be the group generated by all symbols $\; \, g\otimes h, \; g,h\in G$, subject to the relations
\[
gg_1 \otimes h = ( g^{g_1}\otimes h^{g_1}) (g_1\otimes h) \quad
\mbox{and} \quad g\otimes hh_1 = (g\otimes h_1)( g^{h_1} \otimes
h^{h_1})
\]
for all $g,g_1, h,h_1 \in G$, where  we write $x^y$ for the conjugate $y^{-1} x y$ of $x$ by $y$, for any elements $x, y \in G$. In \cite{BL}, Brown and Loday show that the third homotopy group of the suspension of an Eilenberg-MacLane space $K(G,1)$ satisfies $\pi_3(SK(G,1)) \cong \mu(G),$ where $\mu(G)$ denotes the kernel of the derived map $\rho': G \otimes G \to G'$, given by  $g \otimes h \mapsto [g,h]$.

In \cite{NR1}, Rocco considered the following construction. Let $G$ be a group and let $\varphi : G \rightarrow G^{\varphi}$ be an isomorphism ($G^{\varphi}$ is a copy of $G$, where $g \mapsto g^{\varphi}$, for all $g \in G$). Define the group $\nu(G)$ to be \[ \nu (G):= \langle 
G \cup G^{\varphi} \ \vert \ [g_1,{g_2}^{\varphi}]^{g_3}=[{g_1}^{g_3},({g_2}^{g_3})^{\varphi}]=[g_1,{g_2}^{\varphi}]^{{g_3}^{\varphi}},
\; \ g_i \in G \rangle .\]

The motivation for studying $\nu(G)$ is the commutator connection: indeed, the map  $\Phi: G \otimes G \rightarrow [G, G^{\varphi}]$, defined by $g \otimes h \mapsto [g , h^{\varphi}]$ is an isomorphism  \cite[Proposition 2.6]{NR1}. By \cite[Remark 4, pags. 1984--1985]{NR2}, \[\frac{\chi(G)}{R} \cong \frac{\nu(G)}{\Delta(G)},\]   where $\Delta(G) = \langle [g,g^{\varphi}]  \, | \, g \in G\rangle \leq \nu(G)$. In the group $\nu(G)$ consider the following three subgroups: 

\begin{itemize}
    \item $\Upsilon_1(G) = \langle [g,h^{\varphi}] \ \mid \ g,h \in G \rangle = [G,G^{\varphi}]$.
    \item $\Upsilon_2(G) = \langle [h,g][g,h^{\varphi}] \ \mid \ g,h \in G \rangle$.
    \item $\Upsilon_3(G) = \langle [h^{\varphi},g^{\varphi}] [g,h^{\varphi}] \ \mid \ g,h \in G \rangle$.
\end{itemize}

In \cite{BdOMR}, the authors obtain the following description to the derived subgroup $\nu(G)'$: 

\begin{lem}\cite[Theorem A]{BdOMR} \label{lem:upsilon}
Let $G$ be a group. 
\begin{itemize}
    \item[(i)] Then $\nu(G)'$ is a central product of $\Upsilon_1(G)$, $\Upsilon_1(G)$ and $\Upsilon_3(G)$. 
    \item[(ii)] $\Upsilon_i(G)/\Delta(G)$ is isomorphic to the non-abelian exterior product $G\wedge G$. 
\end{itemize}
\end{lem}

Now, we can rephrase the previous result in the context of $\chi(G)/R$: 

\begin{lem} \label{lem:exterior}
Let $G$ be a group. 
\begin{itemize}
    \item[(i)] The group $\chi(G)'/R$ is a central product of $D/R$, $L_1/R$ and $L_2/R$. 
    \item[(ii)] The subgroups $D/R$, $L_1/R$ and $L_2/R$ are isomorphic to the non-abelian exterior product $G \wedge G$.
\end{itemize}
\end{lem}
\begin{proof}
(1) and (2). According to \cite[Theorem 2.1~(ii)]{NR2}, we deduce that: 
\[ 
\dfrac{\nu(G)}{\Delta(G)} \cong \dfrac{\chi(G)}{R}.
\]    
By \cite[Section 2]{NR2}, $D/R$ is isomorphic to the exterior square $G \wedge G$. By \cite[Theorem A~(a)]{BdOMR}, the subgroups $L_1/R$ and $L_2$ are isomorphic to $G \wedge G$.     

Moreover, according to Lemma \ref{lem:upsilon}~(i), we deduce that the derived group $\nu(G)'/\Delta(G)$ is a central product of $D/R$, $L_1/R$ and $L_2/R$. 
\end{proof}

\section{Exponent \texorpdfstring{\(\exp(R(G))\)}{exp(R(G))}}

In this section, we prove Theorem A and provide additional information on the subgroup $L_{1,2}$ when $G$ is a group with restrictions on its set of commutators.

\subsection{Proof of Theorem A}

First we present some information on the subgroup $L_{1,2}$.

\begin{thm}\label{TL12FG}
Let $G$ be a group. 
\begin{enumerate}[(a)]
    \item $L_{1,2}\leq R$.
    \item Suppose that $G$ and $G'$ are $d$-generated, then $L_{1,2}$ has $d$-bounded rank. 
\end{enumerate}
\end{thm}
\begin{proof}
\begin{enumerate}
    \item[(a)] By Lemma \ref{lem:exterior} item (i), we have that $L_1/R$ commutes with $L_2/R$. Therefore, $L_{1,2}=[L_1,L_2] \leq R$.     

    \item[(b)] Since $G$ and $G'$ are finitely generated, it follows that the non-abelian exterior square is finitely generated (cf. \cite[Proposition 5.1]{DLT}). Set $D/R= \langle \, [g_1,h_1^{\varphi}]R, \ \ldots \ , [g_n,h_n^{\varphi}]R\, \rangle$. Our aim is to prove that \[
    L_{1,2}= \left\langle \big[[g_i,\varphi,h_i],[g_j,\varphi,h_j^{\varphi}]\big]\ \big|\  1\leq i,j\leq n \right\rangle.
    \]
    According to Lemma \ref{lem:exterior} item (ii), we deduce that
    \[
    \frac{D}{R}\cong \frac{L_1}{R} \cong \frac{L_2}{R}.
    \]
    Therefore we can write $L_1= {\langle [g_i,h_i][h_i,g_i^{\varphi}] \ | \  1\leq i\leq n \rangle R}$ and  $L_2={\langle [g_i^{\varphi},h_i^{\varphi}][h_i^{\varphi},g_i] \ | \  1\leq i\leq n \rangle R}$. Moreover, Lemma \ref{lem:L12prod} gives us 
      \begin{align*}
    L_{1,2} =& \left[ \left\langle [g_i, h_i][h_i, g_i^{\varphi}] \mid 1 \leq i \leq n \right\rangle R,  
            \left\langle [g_j^{\varphi}, h_j^{\varphi}][h_j^{\varphi}, g_j] \mid 1 \leq j \leq n \right\rangle R \right] \\
            =& \left[ \left\langle [g_i, h_i][h_i, g_i^{\varphi}] \mid 1 \leq i \leq n \right\rangle,  
            \left\langle [g_j^{\varphi}, h_j^{\varphi}][h_j^{\varphi}, g_j] \mid 1 \leq j \leq n \right\rangle \right] \\
            =& \left\langle \big[[g_i, \varphi, h_i], [g_j, \varphi, h_j^{\varphi}]\big] \mid 1 \leq i,j \leq n \right\rangle,
    \end{align*}
    as required. \qedhere
\end{enumerate}
\end{proof}

The proof of Theorem A will be divided into the following two results.

\begin{thm} \label{thm:quot}
Let $G$ be a group. 
\begin{enumerate} [(a)] 
    \item If $G$ is periodic, then $R/L_{1,2}$ is periodic. 
    \item If $G$ has finite exponent, then $ \exp\left( R/L_{1,2}\right)$ divides $\exp(G)$. 
\end{enumerate} 
\end{thm}

\begin{proof}
By Lemma \ref{Lm:R/L12}, the quotient $R/L_{1,2}$ is generated by 
\[\big\langle { [g,\varphi, a,b^{\varphi}]} L_{1,2} \, | \, g,a,b\in G\big\rangle.
\]
Then, it is sufficient to consider the elements:  $[g,\varphi, a,b^{\varphi}]L_{1,2}$. Observe that, by Lemma \ref{Lm:R} item (i), 
    \[
[g,\varphi,a,b^{\varphi}]^nL_{1,2}=[[g,\varphi]^n,a,b^{\varphi}]L_{1,2}= [g^n,\varphi,a,b^{\varphi}]L_{1,2},
    \]
for any positive integer $n$. 

(a) and (b). Now, assume that $G$ is periodic (resp., exponent $m$). Arguing as in the previous paragraph, we deduce that $R/L_{1,2}$ is also periodic (resp., $\exp(R/L_{1,2})$ divides $m$), which completes the proof.
\end{proof}

\begin{remar}
It is worth mentioning that Theorem \ref{thm:quot}~(b) cannot be improved. For instance, the group $G = [81, 12]$, which is the $12$th group of order $81$ in the SmallGroup library \cite{GAP4}, satisfies $\exp(G)=3 = \exp(R/L_{1,2})$.    
\end{remar}

\begin{thm}\label{thm:expL12}
Let $G$ be a group. 
\begin{enumerate} [(a)] 
    \item Suppose that for every $x,y \in G$ there exists a positive integer $m=m(x,y)$ such that $[x,y]^m = 1$. Then $L_{1,2}$ is periodic.
    \item Suppose that there exists a positive integer $m$ such that $[x,y]^m=1$ for every $x,y \in G$. Then $\exp(L_{1,2})$ divides $m$. 
\end{enumerate}    
\end{thm}
\begin{proof}
Choose arbitrarily $a,b,g,h\in G$. Set $x= \big[[a,\varphi,b],[g,\varphi,h^{\varphi}]\big]$. 
First, we prove that $x^n = \big[[a,\varphi,b],[g^{\varphi},h^{\varphi}]^n\big]$. We use induction on $n$. For $n=1$, we have
    \begin{align*}
        x&= \big[[a,\varphi,b],[g,\varphi,h^{\varphi}]\big]\\
        &= \big[[a,\varphi,b],[g^{\varphi},h^{\varphi}][h,g^{\varphi}]\big] \ \ \ \ ( Lemma \ \ref{l1}\ item \ (iv))\\
        &= \big[[a,\varphi,b],[h,g^{\varphi}]\big] \big[[a,\varphi,b],[g^{\varphi},h^{\varphi}]\big]^{[h,g^{\varphi}]}\\
        &= \big[[a,\varphi,b],[g^{\varphi},h^{\varphi}]\big]^{[h,g^{\varphi}]} \ \ \ \ \big(\big[[a,\varphi,b],[h,g^{\varphi}]\big]\in [D,L]=1 \big) \\
        &= \big[[a,\varphi,b],[g^{\varphi},h^{\varphi}]\big] \ \ \ \big( \big[[a,\varphi,b],[g^{\varphi},h^{\varphi}]\big]\in R\big).
    \end{align*}
 Now, let us suppose that the identity holds for $n$. 
    Then,
      \begin{align*}
        x^{n+1}&= \big[[a,\varphi,b],[g,{\varphi},h]\big]^{n+1}\\
         &=\big[[a,\varphi,b],[g,{\varphi},h]\big]^{n}\big[[a,\varphi,b],[g,{\varphi},h]\big]\\
        &=\big[[a,\varphi,b],[g,{\varphi},h]^{n}\big]\big[[a,\varphi,b],[g,{\varphi},h]\big] \\
        &=\big[[a,\varphi,b],[g^{\varphi},h^{\varphi}]^n[g^{\varphi},h^{\varphi}]\big] 
      \ \ \ (Lemma \ \ref{lem:L12prod})\\
        &=\big[[a,\varphi,b],[g^{\varphi},h^{\varphi}]^{n+1}\big]
    \end{align*}

Consequently, the order of an element $x$ divides $m$, where $m$ is the order of the commutator $[g,h]$. 

(a) Suppose that for every $x,y \in G$ there exists a positive integer $m=m(x,y)$ such that $[x,y]^m = 1$. Since $L_{1,2}$ is abelian and $a,b,g,h$ have been chosen arbitrarily, we now conclude that $L_{1,2}$ is periodic.

(b) Suppose that $[x,y]^m=1$ for every $x,y \in G$. Arguing as in the previous paragraph the exponent $\exp(L_{1,2})$ divides $m$. 
\end{proof}

\begin{remar}
The bound obtained in Theorem \ref{thm:expL12}~(b) is somehow sharp. Indeed, the group $G = [243,37]$, which is the $37$th group of order $243$ in the SmallGroup library \cite{GAP4}, satisfies $\exp(G') = 3 = \exp(L_{1,2})$. See also Remark \ref{rem:exp} below.     
\end{remar}

Now we will deal with Theorem A: {\it Let $G$ be a group. 

\begin{enumerate} [(a)]
    \item If $G$ is periodic, then $R$ is periodic.
    \item If $G$ has finite exponent, then $\exp(R)$ divides $\exp(G) \cdot \exp(G')$.
\end{enumerate}
}
\begin{proof}[Proof of Theorem A]
\begin{enumerate}[(a)]
    \item Assume that $G$ is periodic. According to Theorem \ref{thm:quot}~(a) and Theorem \ref{thm:expL12}~(a), we deduce that $R/L_{1,2}$ and $L_{1,2}$ are periodic. Consequently, $R$ is periodic. 
    \item By Theorem \ref{thm:quot}~(b), $\exp(R/L_{1,2})$ divides $\exp(G)$. Theorem \ref{thm:expL12}~(b) now shows that $\exp(L_{1,2})$ divides $\exp(G')$. It follows that $\exp(R)$ divides $\exp(G) \cdot \exp(G')$.  \qedhere
\end{enumerate}
\end{proof}

\begin{remar}
\label{rem:example}
We do not know whether the bound obtained in Theorem A~(b) is sharp, as our proof does not allow us to go further. At this moment, we have no examples that attain the bound in Theorem A~(b). For instance, if $G$ is the $12$th group of order $81$ in the SmallGroup library \cite{GAP4}, then $\exp(G) = 3 = \exp(R)$.
\end{remar}

In \cite{BdMdOM}, the authors provide bounds for the exponent $\exp(R(G))$ of a solvable group $G$ in terms of its derived length. By combining Theorem A and \cite[Theorem 1.6]{BdMdOM}, one obtains

\begin{cor}
Let $p$ be a prime and $G$ a $p$-group solvable with derived length $d$. Then 
 $$\exp(R(G)) \ \ |\ \ \left\{\begin{array}{rl}
 \min \{\exp(G')\cdot\exp(G), \ 2^d \cdot \exp(G')^{d-1}\} & p=2;\\
 \min \{\exp(G')\cdot\exp(G),\ \exp(G')^{d-1}\} & p>2.
 \end{array}\right.$$    
\end{cor}

\subsection{Finiteness properties of commutators} In this subsection we derived some information on the subgroup $L_{1,2}$, when $G$ is a group subject to restrictions on its set of commutators. More precisely, Theorem \ref{thm:expL12} reveals a connection between the set of commutators (of $G$) and the subgroup $L_{1,2} \leq \chi(G)$. The groups that appear in this context are related to the famous Burnside problem (and MacDonald's question \cite[Question 13.34]{MK}). 

Let $m,n$ be positive integers. Let $F_m$ be the free group of rank $m$. The Burnside group $B(m,n)$ is given by the factor group $F_m/F^{n}_m$. In \cite{AN}, Adjan and Novikov prove that $B(m,n)$ is infinite for $n$ odd, $n \geq 4381$ (see also, Ol'shanskii \cite{Ol}). We obtain the following related result. 

\begin{cor}
Let $G = B(m,n)$ the Burnside group. Then $L_{1,2}$ is finite.     
\end{cor}

We can refine the previous result as follows. 

\begin{corB} \label{cor:periodic}
Let $G$ be a finitely generated periodic group. Then $L_{1,2}$ is finite. 
\end{corB}

\begin{proof}
As $G$ is periodic, we have $G/G'$ is finite. From  this we deduce that the derived subgroup $G'$ is also finitely generated. Consequently,  by Theorem \ref{TL12FG}~(b), $L_{1,2}$ is finitely generated. Theorem \ref{thm:expL12} now shows that $L_{1,2}$ has finite exponent. Since $L_{1,2}$ is abelian, we conclude that $L_{1,2}$ is finite, as required.     
\end{proof}

It is possible the exponent $\exp(L_{1,2})$ may be finite while the derived subgroup $G'$ is non-periodic. 

\begin{remar} \label{rem:exp}
In \cite{DK}, Deryabina and Kozhevnikov shows that there exist a positive integer $n$ and a group $K$ in which $[x,y]^n=1$ for every $x,y \in K$ and the derived subgroup is non-periodic (see also, Adjan \cite{A}). By Theorem \ref{thm:expL12}~(b), the exponent $\exp(L_{1,2}(K))$ divides $n$.  
\end{remar}

\section{Some proprieties of the derived and lower central series of \texorpdfstring{$\chi(G)$}{χ(G)}}

\subsection{Derived series of $\chi(G)$}

According to \cite[Proposition 4.1.4]{Sidki}, we deduce that the derived subgroup of $\chi(G)$ satisfies $\chi(G)' \leqslant D(G)L(G)$. In this subsection we describe the derived series of $\chi(G)$ as the central product of the corresponding terms of the derived series of the subgroups $D$, $L_1$, $L_2$ (and $L_{1,2}$). For reader's convenience we restate Theorem C.

\begin{thmC}
Let $G$ be a group.   
\begin{enumerate}[(a)]
    \item The derived subgroup $\chi(G)'$ is the central product of  $D$ and  $L_1L_2$.
    \item The second derived subgroup $\chi(G)^{(2)}$ is given by the product of subgroups $D'L_1'L_2'L_{1,2}$.
    \item If $k \geq 2$, then $\chi(G)^{(k+1)}$ is the central product of $D^{(k)}$, $ L_1^{(k)}$ and $L_2^{(k)}$.
\end{enumerate}
\end{thmC}
\begin{proof} 
\begin{enumerate}[(a)]
    \item By \cite[Proposition 4.1.4]{Sidki}, we can argue that $\chi(G)'= D L_1 L_2$ and $D\cap L_1L_2\leq D\cap L=W\leq C_{\chi(G)}(\chi(G)')$. Therefore, $\chi(G)'$ is the central product of $D$ and $L_1 L_2$.
    \item Using commutator calculus and the previous item, we obtain that
    \[ 
    \chi(G)^{(2)}=[D L_1L_2,D L_1L_2]= D' L_1'L_2'L_{1,2}. 
    \]
    \item We argue by induction on $k\geq 3$. Using commutator calculus, we obtain that 
    \begin{align*}    
        \chi(G)^{(3)}&= [D' L_1'L_2'L_{1,2},D' L_1'L_2'L_{1,2}]=[D' L_1'L_2',D' L_1'L_2']\\
        &= D^{(2)}L_1^{(2)}L_2^{(2)} [L'_{1},L'_2]. 
    \end{align*}
    By Lemma \ref{lem:L12prod}  we have that $[L'_{1},L'_2]=1$. Moreover,
    \[
    D^{(2)}\cap L_1^{(2)},D^{(2)}\cap L_2^{(2)},L_1^{(2)}\cap L_2^{(2)}\leq W\leq C_{\chi(G)}(\chi(G)^{(3)}).
    \]
    Suppose that for $k\geq 3$,  
    \[
    \chi(G)^{(k)}= D^{(k-1)}L_1^{(k-1)}L_2^{(k-1)}. 
    \]
    Consequently,
    \[
    \chi(G)^{(k+1)}= D^{(k)}L_1^{(k)}L_2^{(k)} [L_1^{(k-1)},L_2^{(k-1)}]. 
    \]
    Observe that $[L_1^{(k-1)},L_2^{(k-1)}]\leq [L'_1,L'_2]=1$ and 
    \[
    D^{(k)}\cap L_1^{(k)},D^{(k)}\cap L_2^{(k)},L_2^{(k)}\cap L_1^{(k)}\leq W\leq C_{\chi(G)}(\chi(G)^{(k+1)}). 
    \]
    Therefore, if $k\geq 3$, 
    \[
    \chi(G)^{(k)}= D^{(k-1)}L_1^{(k-1)}L_2^{(k-1)}. 
    \]
    And it is a central product as required. \qedhere
\end{enumerate}
\end{proof}

\subsection{Lower central series of \texorpdfstring{$\chi(G)$}{χ(G)}}

We start with the following lemma  related to the lower central series $\gamma_n(\chi(G))$. 

\begin{lem}\label{Lm:Rcontas} Let $ g,a,b \in G$ and $\alpha\in L$. Then 
  \begin{itemize}
      \item[(i)]$[\alpha,a,g,b^{\varphi}]=[\alpha,a,b^{\varphi},g]$;
      \item[(ii)]$[\alpha,a^{\varphi},g^{\varphi},b]=[\alpha,a^{\varphi},b,g^{\varphi}].$ 
  \end{itemize}
\end{lem}
\begin{proof} 
We will prove only the first item, as the second follows similarly.
     \begin{align*}
        [\alpha,a,g,b^{\varphi}]
        &= [[\alpha,a]^{-1}[\alpha,a]^{g},b^{\varphi}]\\
        &= [[\alpha,a]^{-1},b^{\varphi}]^{[\alpha,a]^{g}}[[\alpha,a]^{g},b^{\varphi}]\\
        &= [[\alpha,a]^{-1},b^{\varphi}][[\alpha,a]^{g},b^{\varphi}] \ \ \ \ ([[\alpha,a]^{-1},b^{\varphi}]\in R)\\
        &= [\alpha,a,b^{\varphi}]^{-[\alpha,a]^{-1}}[\alpha,a,(b^{\varphi})^{g^{-1}}]^{g}\\
        &= [\alpha,a,b^{\varphi}]^{-1}[\alpha,a,(b^{\varphi})^{g^{-1}}]^{g} \ \ \ \ \ ([\alpha,a,b^{\varphi}]^{-1}\in R )\\
        &= [\alpha,a,b^{\varphi}]^{-1}[\alpha,a,b^{\varphi}[b^{\varphi},{g^{-1}}]]^{g}\\
        &= [\alpha,a,b^{\varphi}]^{-1}\big([\alpha,a,[b^{\varphi},{g^{-1}}]][\alpha,a,b^{\varphi}]^{[b^{\varphi},{g^{-1}}]}\big)^{g}\\
        &=[\alpha,a,b^{\varphi}]^{-1}[\alpha,a,b^{\varphi}]^{g} \ \ \ \ ([\alpha,a,[b^{\varphi},{g^{-1}}]]\in [L,D]; [\alpha,a,b^{\varphi}]\in R)\\
        &=[\alpha,a,b^{\varphi}]^{-1}[\alpha,a,b^{\varphi}][\alpha,a,b^{\varphi},g]\\
        &=[\alpha,a,b^{\varphi},g]. \qedhere
    \end{align*}
\end{proof}

The statement below includes Theorem D~(a) from the Introduction. 

\begin{lem}\label{[L_1,G]normal} Let $G$ be a group and let $n$ be a positive integer. Then the subgroups $[D,_nG]$, $[L_1,_nG]$ $[L_2,_nG^{\varphi}]$ are normal subgroups $\chi(G)$. 
\end{lem}
\begin{proof}
We will proceed by induction on $n$. For $n=1$, we have that 
\begin{align*}
    [L_1,G]^{\chi(G)}&= [L_1,G]^{DGG^{\varphi}}\\
    &= [L_1,G]^{G^{\varphi}}\\
    &\leq [L_1,G][L_1,G,G^{\varphi}]\\
    &\leq [L_1,G][L_1,G^{\varphi},G] \ \ \ \ (Lemma \ \ref{Lm:Rcontas})\\
    &\leq [L_1,G][R,G]\\
    &\leq [L_1,G].
 \end{align*}
Therefore, $[L_1,G]\lhd \chi(G)$. Analogously, $[L_2,G^{\varphi}]\lhd\chi(G)$. Now, suppose that $[L_1,_nG],[L_2,_nG^{\varphi}]\lhd \chi(G)$. Then, is enough to prove that $[L_1,_{n+1}G]^{G^{\varphi}}\leq [L_1,_{n+1}G]$ and $ [L_2,_{n+1}G^{\varphi}]^{G}\leq [L_2,_{n+1}G^{\varphi}]$. By Lemma \ref{Lm:Rcontas}, 
    \begin{align*}
        [L_1,_{n+1}G]^{G^{\varphi}}&\leq [L_1,_{n+1}G][[L_1,_{n}G],G,{G^{\varphi}}]\\
        &\leq [L_1,_{n+1}G][[L_1,_{n}G],{G^{\varphi}},G]\\
        &\leq [L_1,_{n+1}G].
    \end{align*}
    We can also prove the case of $[L_2,_{n+1}G^{\varphi}]$ analogously. 
    
Now, we prove that $[D,_nG]$ is a normal subgroup of $\chi(G)$. We proceed by induction. First, by \cite[Lemma 4.1.3]{Sidki}, the subgroup $D$ is normal in $\chi(G)$, then
 \[
     [D,G]^{\chi(G)}= [D,G]^{LG}
     = [D,G]^{G}
     = [D^G,G]
     \leq [D,G].
\]
Suppose that it goes for $n$, then 
    \begin{align*}
        [D,_{n+1}G]^{G^{\varphi}}=[D,_{n+1}G]^{G}=\left[[D,_{n}G]^{G},G^{G}\right]=[D,_{n+1}G],
    \end{align*}
which completes the proof. 
\end{proof}

\begin{lem}\label{RnaIntersessao}
    Let $G$ be a group and and let $n$ be a positive integer. Then 
    $[R,_nG]\leq [D,_{n+1}G] \cap [L_1,_{n+1}G] \cap [L_2,_{n+1}G^{\varphi}].$
\end{lem}
\begin{proof}
    Observe that is enough to prove that 
    \[
    R\leq [D,G]\cap [L_1,G]\cap [L_2,G^{\varphi}].
    \]
    Let $g,h,a\in G$. We want to prove that $[\varphi,g,h,a^{\varphi}]$ belongs to each subgroup of the intersection. In fact 
    \begin{align*}
       [\varphi,g,h,a^{\varphi}] &= [[g,h^{\varphi}][h,g],a^{\varphi}]\\
       &= [g,h^{\varphi},a^{\varphi}]^{[h,g]}[h,g,a^{\varphi}]\\
       &= [g,h^{\varphi},a]^{[h,g]}[h,g,a^{\varphi}].
    \end{align*}
    
Once $[D,G]$ is a normal subgroup of $\chi(G)$ and by Hall-Witt Identity, we have that $[h,g,a^{\varphi}]\in [D,G]$. Therefore, $R\leq [D,G]$. 

Now, without loss of generality, we will prove the result for $[L_1,G]$ as the proof for $[L_2,G^{\varphi}]$ follows analogously. 
 \begin{align*}
       [\varphi,g,h,a^{\varphi}] &= [[g,h^{\varphi}][h,g],a^{\varphi}]\\
       &= ([g,h^{\varphi}][h,g])^{-1} ([g,h^{\varphi}][h,g])^{a^{\varphi}}\\
       &= ([g,h^{\varphi}][h,g])^{-1} [g,h^{\varphi}]^{a^{\varphi}}[h,g]^{a^{\varphi}}\\
       &= ([g,h^{\varphi}][h,g])^{-1} [g,h^{\varphi}]^{a}[h,g]^{a(a^{-1}a^{\varphi})}\\
       &= ([g,h^{\varphi}][h,g])^{-1} [g,h^{\varphi}]^{a}[h,g]^a\big[[h,g]^a,(a^{-1}a^{\varphi})\big]\\
       &= ([g,h^{\varphi}][h,g])^{-1} ([g,h^{\varphi}][h,g])^a\big[[h,g]^a,(a^{-1}a^{\varphi})\big]\\
       &= ([g,h^{\varphi}][h,g])^{-1} ([g,h^{\varphi}][h,g])^a\big[[h,g]^a,(a^{-1}a^{\varphi})\big]\\
       &=[\varphi,g,h,a] \big[h^a,g^a,(a^{-1}a^{\varphi})\big].
    \end{align*}

Again, by Hall-Witt Identity and by Lemma \ref{[L_1,G]normal}, $\big[h^a,g^a,(a^{-1}a^{\varphi})\big]\in [L_1,G]$. Consequently, $R\leq [D,G]{\cap [L_1,G]\cap [L_2,G^{\varphi}]}.$
\end{proof}

The following statement includes Theorem D~(b) from the Introduction. 

\begin{prop}\label{EstruturaGamma(chi)}
     Let $G$ be a group. Then
    \[
\gamma_n(\chi(G))= [D,_{n-2}G][L_1,_{n-2}G][L_2,_{n-2}G^{\varphi}]
    \]
    if $n\geq 3$.
\end{prop}

\begin{proof}
    We proceed by induction on $n \geq 3$. For $n=3$,
\begin{align*}
    \gamma_3(\chi(G))&=[DL_1L_2,GL]\\
    &= [D,G][L_1,G][L_2,G][L_1,L][L_2,L]\\
    &= [D,G][L_1,G]R[L_1,L][L_2,L]
\end{align*}
Observe that $[L_1,L]\leq [L_1,G][L_1,G^{\varphi}] $ and $[L_2,L]\leq [L_2,G][L_2,G^{\varphi}]$. Then,
\begin{align*}
    \gamma_3(\chi(G))&= [D,G][L_1,G]R[L_1,L][L_2,L]\\
    &\leq [D,G][L_1,G]R [L_1,G][L_1,G^{\varphi}] [L_2,G][L_2,G^{\varphi}]\\ 
    &\leq [D,G][L_1,G][L_2,G^{\varphi}]R \\
    &\leq [D,G][L_1,G][L_2,G^{\varphi}]. \ \ \ (Lemma \ \ref{RnaIntersessao})
\end{align*}
This is enough to conclude that $\gamma_3(\chi(G))=[D,G][L_1,G][L_2,G^{\varphi}].$

    Now, suppose that holds for $n\geq 3$, then

 \begin{align*}
    \gamma_{n+1}(\chi(G))=&\big[[D,_{n-1}G]\,[L_1,_{n-1}G]\,[L_2,_{n-1}G^{\varphi}],GL\big]\\
    =&[D,_{n}G]\,[L_1,_{n}G]  
     [[L_2,_{n-1}G^{\varphi}],G]\\
     &[[L_1,_{n-1}G],L]\,[[L_2,_{n-1}G^{\varphi}],L].
\end{align*}   

But, 
\[
[L_1,_{n-1}G,L]\leq [L_1,_{n}G]\,[[L_1,_{n-1}G],G^{\varphi}],
\]
and 
\[
[L_2,_{n-1}G^{\varphi},L]\leq [L_2,_{n}G^{\varphi}]\, [[L_2,_{n-1}G^{\varphi}],G].
\]
Additionally, by Lemma \ref{Lm:Rcontas}, $[L_1,_{n-1}G,G^{\varphi}]$ and $[L_2,_{n-1}G^{\varphi},G]$ are subgroups of $  [R,_{n-1}G]$. Finally, 
by Lemma \ref{RnaIntersessao}, $[R,_{n-1}G]\leq [D,_{n}G]$, and we can conclude that 
\[
\gamma_{n+1}(\chi(G))=[D,_{n-1}G][L_1,_{n-1}G][L_2,_{n-1}G^{\varphi}].  \qedhere
\] 
\end{proof}

\subsection{Higher terms of the lower central series of \texorpdfstring{$\chi(G)$}{χ(G)}}

In this subsection, we present an expression for $\gamma_{k}(\chi(G))$, when $G$ is a nilpotent group of class $c$ and $k \geq c+2$ (cf. \cite[Theorem 3.3]{GRS}).

\begin{prop}\label{PropR>D}
    Let $G$ be a nilpotent group of class $c$. Then, $
    [L_1,_{c}G]$ $[L_2,_{c}G^{\varphi}]$, $[D,_{c}G] \leq [R,_{c-1}G]$.
\end{prop}

\begin{proof}First, we will prove that $[L_1,_{c}G] \leq [R,_{c-1}G]$. 
    In fact, let $g,h,a_1, \ldots , a_{c-1}\in G$ and $x=[\varphi,g,h,a_1,\ldots ,a_{c}]\in [L_1,_cG]$. Then, 
    \begin{align*}
        x=& \big[[g,h^{\varphi}][h,g],a_1,\ldots ,a_{c}\big]\\
        =& \big[[g,h^{\varphi},a_1]^{[h,g]}[h,g,a_1],a_2, \ldots ,a_{c}\big]\\
         =& \big[ \, [[g,h^{\varphi},a_1]^{[h,g]},a_2]^{[h,g,a_1]} [h,g,a_1,a_2], a_3, \ldots ,a_{c}\big].
        \end{align*}
    Inductively,  we can obtain that
        \begin{align*}
        x=&\big [ [\ldots[[g,h^{\varphi},a_{1}]^{[h,g]},a_{2}]^{[g,h,a_{1}]}, \ldots ,a_{c-1}]^{[g,h,a_{1}, \ldots ,a_{c-2}]} , a_c \big]^{[g,h,a_{1}, \ldots, a_{c-1}]}\\
        & [h,g,a_1, \ldots ,a_{c-1}, a_{c}]. 
        \end{align*}
        
        Observe that $[h,g,a_1, \ldots ,a_{c-1}]\in \gamma_{c+1}(G)=1$, which implies that we can rewrite $[h,g,a_1, \ldots ,a_{c-1}, a_{c}]$ as $[h,g,a_1, \ldots ,a_{c-1}, a^{\varphi}_{c}]$. Moreover, once $[g,h^{\varphi}]\in D$ we can also rewrite 
        \[
\big [ [\ldots[[g,h^{\varphi},a_{1}]^{[h,g]},a_{2}]^{[g,h,a_{1}]}, \ldots ,a_{c-1}]^{[g,h,a_{1}, \ldots ,a_{c-2}]} , a_c \big]^{[g,h,a_{1}, \ldots, a_{c-1}]},
        \]
        
        as
        \[
\big [ [\ldots[[g,h^{\varphi},a_{1}]^{[h,g]},a_{2}]^{[g,h,a_{1}]}, \ldots ,a_{c-1}]^{[g,h,a_{1}, \ldots ,a_{c-2}]} , a^{\varphi}_c \big]^{[g,h,a_{1}, \ldots, a_{c-1}]}.       
        \]
        Therefore, 
        \begin{align*} x=&\big [ [\ldots[[g,h^{\varphi},a_{1}]^{[h,g]},a_{2}]^{[g,h,a_{1}]}, \ldots ,a_{c-1}]^{[g,h,a_{1}, \ldots ,a_{c-2}]} , a^{\varphi}_c \big]^{[g,h,a_{1}, \ldots, a_{c-1}]} \\
        & [h,g,a_1, \ldots , a_{c-1}, a^{\varphi}_{c}]\\
        =& [\varphi,g,h,a_1,\ldots, a_{c-1},a^{\varphi}_{c}]\\
        =& [\varphi,g,h,a^{\varphi}_{c},a_1,a_2,\ldots, a_{c-1}] \ \ \ (Lemma\  \ref{Lm:Rcontas}).
    \end{align*}
    This is sufficient to conclude that $[L_1,_cG]\leq [R,_{c-1}G]$. The case of $[L_2,_cG^{\varphi}]$ follows analogously.  
Now, let $y=[g^{\varphi},h,a_1, \ldots, a_c]\in [D,_cG]$
    \begin{align*}
       y=& [[\varphi,h,g][g,h], a_1, \ldots, a_c]\\
        =& [[\varphi,h,g, a_1]^{[g,h]}[g,h,a_1],a_2, \ldots, a_c]\\
        =& [[[\varphi,h,g, a_1]^{[g,h]},a_2]^{[g,h,a_1]}[g,h,a_1,a_2],a_3, \ldots, a_c].
    \end{align*}
Inductively, we have that 
    \begin{align*}    
         y =&\big [ [\ldots[[\varphi,h,g, a_1]^{[g,h]},a_2]^{[g,h,a_1]}, \ldots ,a_{c-1}]^{[g,h,a_{1}, \ldots ,a_{c-2}]} , a_c \big]^{[g,h,a_{1}, \ldots, a_{c-1}]} \\
         & [g,h,a_1, \ldots, a_c]\\
        =& \big [ [\ldots[[\varphi,h,g, a_1]^{[g,h]},a_2]^{[g,h,a_1]}, \ldots ,a_{c-1}]^{[g,h,a_{1}, \ldots ,a_{c-2}]} , a_c \big]^{[g,h,a_{1}, \ldots, a_{c-1}]}.
    \end{align*}
Therefore, $[D,_cG]\leq [L_1,_cG]\leq [R,_{c-1}G]$.
\end{proof}

\begin{cor}\label{Cor:GammaCh}
Let $G$ be a nilpotent group of class $c$. Then 
\[
\gamma_{n+2}(\chi(G)) = [R,_{n-1}G] = [D,_nG]=[L_1,_nG]=[L_2,_nG^{\varphi}],
\]
for every $n\geq c$.
\end{cor}

\begin{proof}
Combining Lemma \ref{RnaIntersessao}, Proposition \ref{EstruturaGamma(chi)} and Proposition \ref{PropR>D} we conclude the proof.
\end{proof}

\begin{remar}
According to \cite[Theorem 3.3]{GRS}, we deduce that $\gamma_{c+3}(\chi(G))$ is elementary abelian $2$-group. Applying Corollary \ref{Cor:GammaCh} to $\gamma_{c+3}(\chi(G))$, we obtain that \[\gamma_{c+3}(\chi(G))^2= [R,_c G]^2=1.\]   
\end{remar}

\section*{Acknowledgements}

We are grateful to Luis Mendon\c ca and Said Sidki for useful discussions on weak commutativity. This work was partially supported by FAPDF, CAPES and CNPq. The first author acknowledges the support of the CNPq projects \textit{Produtividade em Pesquisa} (Project No.: 303191/2022-8) and \textit{P\'os-doutorado no Exterior} (Project No.: 200116/2025-8). The last author was financed in part by the CAPES -- Finance Code 001. 

\end{document}